\documentclass[10pt,reqno]{amsart}

\usepackage[utf8]{inputenc}
\usepackage{amsfonts}
\usepackage{amsmath}
\usepackage{mathrsfs}
\usepackage{amsthm}
\usepackage{amssymb}
\usepackage[top=30truemm,bottom=30truemm,left=30truemm,right=30truemm]{geometry}
\usepackage[bookmarks=false,draft=false,breaklinks,colorlinks]{hyperref}
\usepackage[dvipdfmx]{graphicx}

\usepackage{color}
\usepackage{url}

\usepackage{fancyhdr}
\usepackage{enumitem}
\setenumerate{label=(\arabic*),nosep}
\setitemize{nosep}
\newlist{clist}{enumerate}{1}
\setlist*[clist]{label=(\roman*),nosep}

\theoremstyle{definition}
\newtheorem{Def}{Definition}[section]
\newtheorem{Thm}[Def]{Theorem}

\newtheorem{Lem}[Def]{Lemma}
\newtheorem{Cor}[Def]{Corollary}

\newcommand{\C}{\mathbb{C}} 
\newcommand{\N}{\mathbb{N}} 

\title{An operator-coefficients free Poincar\'{e} inequality}
\author{Hyuga Ito}

\address{
Graduate School of Mathematics, Nagoya University, Furocho, Chikusaku, Nagoya, 464-8602, Japan
}
\email{hyuga.ito.e6@math.nagoya-u.ac.jp}

\date{\today}

\begin{document}

\maketitle

\begin{abstract}
 We prove the following operator-coefficients free Poincar\'{e} inequality:
 \[
 |f(X)-E[f(X)]|_{2}\leq2|X|_{2}\|\widehat{\partial}_{X:B}[f(X)]\|_{\pi}, \quad f(X) \in \mathrm{dom}(\widehat{\partial}_{X:B}),
 \]
 where $\|\cdot\|_{\pi}$ is the projective tensor norm.
\end{abstract}

\allowdisplaybreaks{

\section{Introduction}
Let $M$ be a von Neumann algebra with a faithful normal tracial state $\tau$ on $M$, and $B$ be a unital von Neumann subalgebra of $M$ with a (unique) $\tau$-preserving conditional expectation $E$ from $M$ onto $B$. Let $X$ be a self-adjoint elememt of $M$, which is assumed to be algebraically free from $B$. Let $B\langle X\rangle$ denote the family of all $B$-valued non-commutative polynomials, that is, 
\begin{equation*}
    B\langle X\rangle
    =\mathrm{span}\{b_{0}Xb_{1}X\dots Xb_{n} |n\in\N,\,b_{i}\in B\},
\end{equation*}
and $\mu$ denotes the usual multiplication on $B\langle X\rangle$. 
The free difference quotient $\partial_{X:B}:B\langle X\rangle\to B\langle X\rangle^{\otimes2}$ is the unique $B\langle X\rangle^{\otimes2}$-valued derivation on $B\langle X\rangle$ that satisfies 
$\partial_{X:B}[X]=1\otimes1$ and $\partial_{X:B}[b]=0$ for any $b\in B$. 
Let $L^2(M,\tau)=L^2(M)$ denote the completion of $M$ with respect to the (tracial) $L^2$-norm defined by $|a|_{2}=\tau(a^*a)^{\frac{1}{2}}$ for every $a\in M$. 
Set $B\langle t\rangle:=B*\C\langle t\rangle$ (algebraic free product) with indeterminate $t$. Note that any element of $B\langle t\rangle$ is a linear combination of monomials $b_{0}tb_{1}t\cdots tb_{n}$ $(b_{i}\in B)$. For any $R>0$, let $B_{R}\{t\}$ be the completion of $B\langle t\rangle$ with respect to the norm $\||\cdot\||_{R}$ defined by 
\begin{equation*}
    \||p(t)\||_{R}=\inf\left\{\sum_{k=1}^{n}\|b_{k,0}\|\cdot\|b_{k,1}\|\cdots\|b_{k,m(k)}\|R^{m(k)}\middle|p(t)=\sum_{k=1}^{n}b_{0}tb_{1}\cdots tb_{m(k)}\right\}
\end{equation*}
for every $p(t)\in B\langle t\rangle$.

The purpose of this short note is to give, with a very simple proof, an operator-coefficients free Poincar\'{e} inequality, which is almost the same as what Voiculescu conjectured (see \cite{aim06}) but we choose the norm of $\partial_{X:B}[p(X)]$ here to be the projective tensor norm instead of the $L^2$-norm. We remark that a scalar-coefficients free Poincar\'{e} inequality has successfully been established by Voiculescu in his unpublished note, and its proof can be found in e.g. \cite[section 8.1]{misp17}.

\medskip
The author acknowledges his supervisor Professor Yoshimichi Ueda for conversations, comments and editorial supports to this note. The author also expresses his sincere gratitude to Professor Dan Voiculescu who kindly gave him and his supervisor an additional explanation to \cite{voi00}.

\section{Main result}\label{pre} 
 In this section, we assume that $\partial_{X:B}$ from $(C^*(B\langle X\rangle),\|\cdot\|)$ to $(C^*(B\langle X\rangle)^{\overline{\otimes}2},\|\cdot\|)$ is closable (this follows from the existence of conjugate variable in $L^2(M)$, see \cite[Corollary 4.2]{voi98} and \cite[section 3.2]{voi00}), and $\overline{\partial}_{X:B}$ denotes the closure of $\partial_{X:B}$ with respect to $\|\cdot\|$ on both sides.
Also, it is true that the natural map from the projective tensor product $C^*(B\langle X\rangle)^{\widehat{\otimes}2}\subset M^{\widehat{\otimes}2}$ into $C^*(B\langle X\rangle)^{\overline{\otimes}2}\subset M^{\overline{\otimes}2}$ is injective. This indeed follows from Haagerup's famous work \cite[Proposition 2.2]{ha85}. Hence, $\partial_{X:B}$ from $(C^*(B\langle X\rangle),\|\cdot\|)$ to $(C^*(B\langle X\rangle)^{\widehat{\otimes}2},\|\cdot\|_{\pi})$ is closable if so is $\partial_{X:B}$ from $(C^*(B\langle X\rangle),\|\cdot\|)$ to $(C^*(B\langle X\rangle)^{\overline{\otimes}2},\|\cdot\|)$ and $\widehat{\partial}_{X:B}$ denotes the closure of $\partial_{X:B}$ with respect to $\|\cdot\|$ and $\|\cdot\|_{\pi}$.

Voiculescu introduced a certain smooth subalgebra of $C^*(B\langle X\rangle)$, which is a kind of Sobolev space (see \cite[section 4]{voi98}). Let $B^{(1)}(X)$ be the completion of $B\langle X\rangle$ with respect to the norm $\||\cdot\||_{(1)}$ defined by 
\begin{equation*}
    \||p(X)\||_{(1)}:=\|p(X)\|+\|\partial_{X:B}[p(X)]\|_{\pi}
\end{equation*}
for any $p(X)\in B\langle X\rangle$. The resulting space becomes a Banach$*$-algebra. Then, we can show the next two lemmas. 

\begin{Lem}\label{sob}
    We have the following facts:
    \begin{enumerate}
        \item\label{sob1} For any $\eta\in B^{(1)}(X)$ there exist $\eta_{\pi}\in C^*(B\langle X\rangle)^{\widehat{\otimes}2}$ and $\eta_{\infty}\in C^*(B\langle X\rangle)$ such that $\||\eta\||_{(1)}=\|\eta_{\infty}\|+\|\eta_{\pi}\|_{\pi}$.
        \item\label{sob2} The correspondence $\iota:B^{(1)}(X)\to C^*(B\langle X\rangle)$ given by $\iota[\eta]:=\eta_{\infty}$ for every $\eta\in B^{(1)}(X)$ defines a contractive algebra homomorphism with $\iota|_{B\langle X\rangle}=\mathrm{id}_{B\langle X\rangle}$. With this map, we regard $B^{(1)}(X)$ as a $*$-subalgebra of $C^*(B\langle X\rangle)$.
        \item\label{sob3} The correspondence $\widetilde{\partial}_{X:B}:B^{(1)}(X)\to C^*(B\langle X\rangle)^{\widehat{\otimes}2}$ given by $\widetilde{\partial}_{X:B}[\eta]:=\eta_{\pi}$ for every $\eta\in B^{(1)}(X)$ defines a contractive derivation and satisfies $\widetilde{\partial}_{X:B}=\widehat{\partial}_{X:B}\circ\iota$ and hence $\widetilde{\partial}_{X:B}|_{B\langle X\rangle}=\partial_{X:B}$.
        \item\label{sob4} The non-commutative functional calculus map $f(t)\mapsto f(X)$ from $B_{R}\{t\}$ to $C^*(B\langle X\rangle)$ sending $t$ to $X$ is well defined as long as $\|X\|<R$, and its range becomes a $*$-subalgebra of $B^{(1)}(X)$.
    \end{enumerate}
\end{Lem}
\begin{proof}
    We give only a sketch of proof.
    
    (1) This follows from the definition of norm $\||\cdot\||_{(1)}$.

    (2) The well-definedness of $\iota$ follows from that for any $\eta\in B^{(1)}(X)$, $\eta_{\infty}$ does not depend on the choice of approximation sequence of $\eta_{\infty}$.

    (3) The well-definedness of $\widetilde{\partial}_{X:B}$ follows similarly to \ref{sob2}. By the construction of $\widetilde{\partial}_{X:B}$ and the closability of $\widehat{\partial}_{X:B}$, we have $\widetilde{\partial}_{X:B}=\widehat{\partial}_{X:B}\circ\iota$. Showing that $\widetilde{\partial}_{X:B}$ is a derivation needs only the first part of \cite[Lemma 3.1]{voi00}, whose discussion works in the present setting. 

    (4) Use the following inequalities (see \cite[section 4]{voi98}):
    \begin{equation*}
        \|p(X)\|\leq\||p(t)\||_{R},\quad \|\partial_{X:B}[p(X)]\|\leq\|\partial_{X:B}[p(X)]\|_{\pi}\leq C\||p(t)\||_{R}
    \end{equation*}
    for any $p(t)\in B\langle t\rangle$, where $C=\sup_{n\in\N}\frac{n\|X\|^{n-1}}{R^n}$.
\end{proof}

\begin{Lem}\label{sobd}
    The map $\iota:B^{(1)}(X)\to C^*(B\langle X\rangle)$ is injective. Moreover, the range of $\iota$ is exactly $\mathrm{dom}(\widehat{\partial}_{X:B})$. 
\end{Lem}
\begin{proof}
    We show only the second part. By Lemma \ref{sob}\ref{sob3}, that $\mathrm{ran}(\iota)\subset\mathrm{dom}(\widehat{\partial}_{X:B})$ holds. Conversely, for any $f(X)\in\mathrm{dom}(\widehat{\partial}_{X:B})$ there exists a sequence $\{p_{n}(X)\}_{n=1}^{\infty}\subset B\langle X\rangle$ such that $p_{n}(X)\xrightarrow{n\to\infty}f(X)$ in $\|\cdot\|$ and $\partial_{X:B}[p_{n}(X)]\xrightarrow{n\to\infty}\widehat{\partial}_{X:B}[f(X)]$ in $\|\cdot\|_{\pi}$. Then, we have
    \begin{gather*}
        \||p_{n}(X)-p_{m}(X)\||_{(1)}
        =\|p_{n}(X)-p_{m}(X)\|+\|\partial_{X:B}[p_{n}(X)]-\partial_{X:B}[p_{m}(X)]\|_{\pi}\\
        \xrightarrow{n\to\infty}\|f(X)-f(X)\|+\|\widehat{\partial}_{X:B}[f(X)]-\widehat{\partial}_{X:B}[f(X)]\|_{\pi}=0. 
    \end{gather*}
    Therefore, there exists an $\eta\in B^{(1)}(X)$ such that $p_{n}(X)\xrightarrow{n\to\infty}\eta$ in $\||\cdot\||_{(1)}$ and we have $f(X)=\iota[\eta]$. Thus, $\mathrm{dom}(\widehat{\partial}_{X:B})\subset\mathrm{ran}(\iota)$.
\end{proof}

We are now in a position to prove the main theorem.

\begin{Thm}[An operator-coefficients free Poincar\'{e} inequality]\label{ovfpi}
We have that for any $f(X)\in\mathrm{dom}(\widehat{\partial}_{X:B})$
     \begin{equation*}
      |f(X)-E[f(X)]|_{2}\leq 2|X|_{2}\|\widehat{\partial}_{X:B}[f(X)]\|_{\pi},
     \end{equation*}
 or equivalently, by Lemma \ref{sobd}, for any $f(X)\in B^{(1)}(X)$, the same inequality also holds with $\widetilde{\partial}_{X:B}[f(X)]$ in place of $\widehat{\partial}_{X:B}[f(X)]$, where $\|\cdot\|_{\pi}$ is the projective tensor norm on $C^*(B\langle X\rangle)^{\widehat{\otimes}2}$.
\end{Thm}

\begin{proof}
 By the continuity of $E$ and of norm, it suffices to show the inequality for any non-commutative polynomial $p(X)\in B\langle X\rangle$ (in this case, of course, $\partial_{X:B}[p(X)]=\widehat{\partial}_{X:B}[p(X)]=\widetilde{\partial}_{X:B}[p(X)]$), and denote by $\mu$ the multiplication map from $B\langle X\rangle^{\otimes2}$ to $B\langle X\rangle$. 
 Let $\sharp$ be a bilinear map on $B\langle X\rangle^{\otimes2}$ such that $(a_{1}\otimes a_{2})\sharp(a_{3}\otimes a_{4})=(a_{1}a_{3})\otimes(a_{4}a_{2})$ for every $a_{i}\in B\langle X\rangle$. For any $p(X)\in B\langle X\rangle$ and any expression $\partial_{X:B}[p(X)]=\sum_{i=1}^{N}q_{i,1}(X)\otimes q_{i,2}(X)\in B\langle X\rangle^{\otimes2}$ with monomials $q_{i,j}(X)$, we have 
 \begin{align*}
     (\mu\circ(\mathrm{id}\otimes E))(\partial_{X:B}[p(X)]\sharp(X\otimes1-1\otimes X))
     &=\sum_{i=1}^{N}(q_{i,1}(X)XE[q_{i,2}(X)]-q_{i,1}(X)E[Xq_{i,2}(X)]).
 \end{align*}
 On the other hand, for any monomial $q(X)=b_{0}Xb_{1}\cdots Xb_{n}\in B\langle X\rangle$, we have
 \begin{align*}
     \partial_{X:B}[q(X)]\sharp(X\otimes1-1\otimes X)
     &=\left(\sum_{i=1}^{n}b_{0}Xb_{1}\cdots b_{i-1}\otimes b_{i}X\cdots Xb_{n}\right)\sharp(X\otimes1-1\otimes X)\\
     &=b_{0}X\otimes b_{1}\cdots Xb_{n}-b_{0}\otimes Xb_{1}\cdots Xb_{n}\\
     &\quad+b_{0}Xb_{1}X\otimes b_{2}\cdots Xb_{n}-b_{0}Xb_{1}\otimes Xb_{2}\cdots Xb_{n}\\
     &\quad+b_{0}Xb_{1}Xb_{2}X\otimes b_{3}\cdots Xb_{n}-b_{0}Xb_{1}Xb_{2}\otimes Xb_{3}\cdots Xb_{n}\\
     &\qquad\qquad\vdots\\
     &\quad+b_{0}Xb_{1}X\cdots b_{n-1}X\otimes b_{n}
     -b_{0}Xb_{1}X\cdots b_{n-1}\otimes Xb_{n}.
 \end{align*}
 Since $E$ is a $B$-bimodule map, it follows that
 \begin{align*}
    (\mu\circ(\mathrm{id}\otimes E))(\partial_{X:B}[q(X)]\sharp(X\otimes1-1\otimes X))
    &=b_{0}X E[b_{1}\cdots Xb_{n}]-b_{0} E[Xb_{1}\cdots Xb_{n}]\\
    &\quad+b_{0}Xb_{1}X E[b_{2}\cdots Xb_{n}]-b_{0}Xb_{1} E[Xb_{2}\cdots Xb_{n}]\\
    &\quad+b_{0}Xb_{1}Xb_{2}X E[b_{3}\cdots Xb_{n}]-b_{0}Xb_{1}Xb_{2} E[Xb_{3}\cdots Xb_{n}]\\
    &\qquad\qquad\vdots\\
    &\quad+b_{0}Xb_{1}X\cdots b_{n-1}X E[b_{n}]
    -b_{0}Xb_{1}X\cdots b_{n-1} E[Xb_{n}]\\
    &=b_{0}X E[b_{1}\cdots Xb_{n}]-E[q(X)]\\
    &\quad+b_{0}Xb_{1}X E[b_{2}\cdots Xb_{n}]-b_{0}XE[b_{1}Xb_{2}\cdots Xb_{n}]\\
    &\quad+b_{0}Xb_{1}Xb_{2}XE[b_{3}\cdots Xb_{n}]-b_{0}Xb_{1}XE[b_{2}Xb_{3}\cdots Xb_{n}]\\
    &\qquad\qquad\vdots\\
    &\quad+q(X)
     -b_{0}Xb_{1}X\cdots XE[b_{n-1}Xb_{n}]\\
    &=q(X)-E[q(X)].
 \end{align*}
 By linearlity, we obtain
 \begin{equation*}
     (\mu\circ(\mathrm{id}\otimes E))(\partial_{X:B}[p(X)]\sharp(X\otimes1-1\otimes X))=p(X)-E[p(X)]
 \end{equation*}
 for any $p(X)\in B\langle X\rangle$. 
Therefore, we have
 \begin{align*}
     |p(X)-E[p(X)]|_{2}
     &=|(\mu\circ(\mathrm{id}\otimes E))(\partial_{X:B}[p]\sharp(X\otimes1-1\otimes X))|_{2}\\
     &=\left|\sum_{i=1}^{N}(q_{i,1}(X)XE[q_{i,2}(X)]-q_{i,1}(X)E[Xq_{i,2}(X)])\right|_{2}\\
     &\leq\sum_{i=1}^{N}(|q_{i,1}(X)XE[q_{i,2}(X)]|_{2}+|q_{i,1}(X)E[Xq_{i,2}(X)]|_{2})\\
     &\leq2|X|_{2}\sum_{i=1}^{N}\|q_{i,1}(X)\|\cdot\|q_{i,2}(X)\|
 \end{align*}
 since $\tau$ is tracial and $E$ is contractive. It follows that
 \begin{equation*}
     |p(X)-E[p(X)]|_{2}\leq2|X|_{2}\|\partial_{X:B}[p(X)]\|_{\pi}
 \end{equation*}
 by the definition of projective tensor norm.
\end{proof}

The inequality still holds even if the $L^2$-norm is replaced with the operator norm. The proof is completely identical.

\begin{Cor}\label{ker} Both the kernels of $\widehat{\partial}_{X:B}$ and $\widetilde{\partial}_{X:B}$ are exactly $B$.
\end{Cor}

 By $\|\partial_{X:B}[p(X)]\|\leq \|\partial_{X:B}[p(X)]\|_{\pi}$ for every $p(X)\in B\langle X\rangle$, Lemma \ref{sob}\ref{sob4} and Lemma \ref{sobd}, we have $\{f(X)\,|\,f(t)\in B_{R}\{t\}\}\subset B^{(1)}(X)=\mathrm{dom}(\widehat{\partial}_{X:B})\subset\mathrm{dom}(\overline{\partial}_{X:B})$ when $\|X\|<R$ and $\overline{\partial}_{X:B}$ is an extension of $\widehat{\partial}_{X:B}$ (via the natural injection from $M^{\widehat{\otimes}2}$ to $M^{\overline{\otimes}2}$ due to \cite[Proposition 2.2]{ha85}). Therefore, Corollary \ref{ker} shows that the kernel of the restriction of $\overline{\partial}_{X:B}$ to $B^{(1)}(X)$ is $B$. This is an improvement of \cite[Lemma 3.4]{voi00}.

}

\end{document}